\documentclass[1p]{elsarticle}

\usepackage{amssymb}
\usepackage{amsthm}
\usepackage{mathrsfs,amsmath}
\usepackage{dsfont}

\theoremstyle{definition}

\newtheorem{theorem}{Theorem}[section]
\newtheorem{definition}[theorem]{Definition}
\newtheorem{lemma}[theorem]{Lemma}
\newtheorem{corollary}[theorem]{Corollary}
\newtheorem{proposition}[theorem]{Proposition}

\newtheorem{acknowledgments}[theorem]{Ackowledgments}

\newtheorem{problem}[theorem]{Problem}

\DeclareMathOperator{\Con}{Con}
\DeclareMathOperator{\Aut}{Aut}

\DeclareMathOperator{\Pol}{Pol}
\DeclareMathOperator{\Cg}{Cg}

\DeclareMathOperator{\id}{id}

\begin{document}

\title{Gr\"{u}n's Lemma for Semiassociative Mal'cev Algebras}
\author{Alexander Wires}
\ead{awires81@uestc.edu.cn}
\address{School of Mathematical Sciences, University of Electronic Science and Technology of China, Chengdu, 611731, Sichuan, PRC}

\begin{abstract}
We extend recent work on analogues of Gr\"{u}n's Lemma from skew-braces to all semiassociative Mal'cev algebras which is a broad class of algebras encompassing all expansions of groups. The direct analogue of Gr\"{u}n's Lemma is known to fail for skew-braces; thus, the more general result is a characterization for perfect algebras of when the upper-central series stabilizes at the second-center. The proof is an application of the description of generating sets for commutators in Mal'cev algebras.
\end{abstract}

\begin{keyword}
Semiassociative Mal'cev algebra \sep perfect algebra \sep second-center \sep skew brace \sep commutator \\
\MSC[2020] 08A30 \sep 20F14 \sep 17B05
\end{keyword}

\maketitle


\section{Introduction}

Grun's Lemma \cite[Satz 4]{Grun} is the result that for a perfect group, the upper central series stabilizes at the second-center; that is, $Z_{2}(G) = Z(G)$. For other algebras, we may say an unconditional Gr\"un's lemma holds if the hypothesis that the algebra is perfect is sufficient to guarantee that the second-center coincides with the center. In the case of skew braces $\left\langle A, \cdot, \circ \right\rangle$, it is known that an unconditional Gr\"un's lemma fails; however, the details are complicated by the fact that skew braces naturally admit competing notions of commutators and centrality due to the presence of the two group operations. Let $A \ast A$ (see \cite[Prop 2.1]{cedoven}) denote the ideal generated by the derived operation $x \ast y = x^{-1} \cdot (x \circ y) \cdot y^{-1}$ and consider the possible candidates for a center given by
\begin{align*}
\mathrm{Soc}(A) &= \{a \in A : \forall x \in A,  a \ast x = 1 \} \cap Z(A,\cdot), \\
Z(A) &= \mathrm{Ann}(A) = \mathrm{Soc}(A) \cap Z(A,\circ),
\end{align*}
the \emph{socle} and \emph{annihilator} (or algebraic center), respectively. It is known from Ced\'o, Smoktunowicz, L. Vendramin \cite[Sec 3]{cedoven} that there are skew braces with $A \ast A = A$, but $\mathrm{Soc}(A/\mathrm{Soc}(A)) \neq 1$; furthermore, Tsang \cite[Sec 3]{Tsang} develops a stronger example of a skew brace with $A \ast A = A$, but $Z(A/Z(A)) \neq 1$.

Interestingly, it is possible to characterize in this case when the second-center and center coincide. Through a sequence of elegant equational derivations (see \cite[Thm 1.6, Cor 2.6]{Tsang}), it is shown that for a skew brace with $A \ast A = A$, the equality $Z_{2}(A) = Z(A)$ can be characterized by the condition that conjugation $i_{a}(x) = a \circ x \circ \bar{a} \in \Aut (A,\cdot)$ for all $a \in Z_{2}(A)$. A more natural choice for the commutator is the ideal $[A,A]^{A}$ generated by $x \ast y$ and the commutators in the two group operations (see Bourn, Facchini, Pompii \cite{SKB}). Recent work of Ballester-Bolinches, Esteban-Romero, Kurdachenko, and Pérez-Altarriba \cite[Thm A]{many1} provide an improvement to the characterization of $Z_{2}(A) = Z(A)$ in \cite{Tsang} by weakening the assumption to skew braces which are \emph{perfect} in the sense that $[A,A]^{A} = A$. Skew braces satisfy a conditional Gr\"un's Lemma in that for perfect skew braces, there is a characterization of when the upper central series stabilizes at the second-center. In this paper, we present a considerable extension to this line of reasoning by proving a conditional Grun's Lemma for all semiassociative Mal'cev algebras (Theorem~\ref{thm:grun}).

Semiassociative Mal'cev algebras form a natural class of algebras encompassing all expanded groups, including multilinear algebras and multiplicative Lie algebras, and less well-studied algebras in this context like conformal algebras by the work of J.D.H. Smith \cite{smith}. They also generate varieties which are important examples of semi-abelian categories, and more generally, protomodular categories which are central objects in categorical algebra (see Borceux and D. Bourn \cite{protomod} and references therein).

The present approach to a general Grun's Lemma is based on that aspect of \cite{Tsang} and \cite{many1} which is closest to the original proof for groups and calculations with the Lie bracket as a generalized commutator operation: (1) a description of generating sets for the commutator and description of the center in terms of binary absorbing terms; (2) elements of the second-center distribute over commutator operations.  Observing this, we are able to efficiently leverage the generators of higher commutators developed for Mal'cev algebras in Aichinger and Mudrinski \cite{supernil} and the basic properties of commutators in Mal'cev varieties. As a preliminary step, in Proposition~\ref{lem:comgensub} we extend to semiassociative Mal'cev algebras the description of the ideal associated to a commutator which was given in \cite[Cor 6.12]{supernil} only for expanded groups. What is gained by an expansive generality in Theorem~\ref{thm:grun} is at the cost of the characterization in Theorem~\ref{thm:grun}(G2) which relies on the choice of generating sets for the commutator; however, though it complicates the preamble of the theorem, it provides a framework for calculations in any particular algebra. As an illustration, groups (Corollary~\ref{cor:grungrp}) and multiplicative Lie algebras (Corollary~\ref{cor:14}) satisfy an unconditional Gr\"un's Lemma. We point out that a preliminary analogue of Grun's Lemma for multiplicative Lie algebras was given in \cite[Prop 2.5]{bak} for notions of partial centrality. The conditional \cite[Thm A]{many1} is rederived (Corollary~\ref{cor:perfgrun}). This provides an illustration for Problem~\ref{prob:2} which situates \cite[Thm A]{many1} (Corollary~\ref{cor:perfgrun}) in a larger context in which the characterization in Grun's Lemma is refined by the calculations germain to the special identities of the particular algebras under focus.


\section{Mal'cev Algebras}\label{section:2}

To set notation, the block of the congruence $\alpha \in \Con A$ which contains the element $u \in A$ is denoted by $u/\alpha$. Membership of a pair $(a,b) \in \alpha$ in a congruence may also be written as $a \mathrel{\alpha} b$; consequently, the notation $\bar{a} \mathrel{\alpha} \bar{b}$ for tuples $\bar{a},\bar{b} \in A^{n}$ means $a_{i} \mathrel{\alpha} b_{i}$ for each coordinate. A \emph{polynomial} $p$ of an algebra $A$ is an operation which results from the partial evaluation of a term; that is, of the form $p(x_{1},\ldots,x_{k}) = t(x_{1},\ldots,x_{k},c_{1},\ldots,c_{n-k})$ for a term $t(x_{1},\ldots,x_{n})$ and some fixed constants $c_{1},\ldots,c_{n-k} \in A$. The set of polynomials of arity $k$ is denoted by $\Pol_{k} A$.

\begin{definition}
An algebra $A$ is a \emph{Mal'cev algebra} if there is a term operation $m(x,y,z)$ satisfying the identities 
\begin{align}
y &= m(y,x,x) & y &= m(x,x,y). 
\end{align}
The Mal'cev algebra is \emph{semiassociative} if the term $m(x,y,z)$ satisfies the additional identity
\begin{align}\label{eqn:1a}
m(m(x,y,z),z,w) &= m(x,y,w). 
\end{align}
\end{definition}

If we fix an element $u \in A$ and define $x +_{u} y := m(x,u,y)$ and $x \mathrel{-_{u}} y := m(x,y,u)$, then the identity Eq~\eqref{eqn:1a} becomes 
\begin{align*}
(x -_{u} y ) +_{u} y = x  \quad \quad \text{ and } \quad \quad (x +_{u} y) -_{u} y = x
\end{align*}
which suggests a form of nonassociative cancellation. For any $a \neq b \in A$ and congruence $\alpha \in \Con A$, the above identities imply the map 
\begin{align*}
a/\alpha \ni x &\longmapsto m(x,a,b) \in b/\alpha 
\end{align*}
is a bijection between the congruence blocks with inverse $x \mapsto m(x,b,a)$; consequently, in semiassociative Mal'cev algebras congruences are completely determined by any of their congruence blocks.

If $A$ has an idempotent element $u \in A$, then for any congruence $\alpha \in \Con A$, the block $u/\alpha$ is a subalgebra called the \emph{ideal associated to} $\alpha$ and carries additional closure properties (consult \cite[Thm 3.23]{protomod} where the terminology of a normal subalgebra is used); conversely, an ideal $I$ determines a congruence $\alpha_{I} = \{ (a,b) : m(a,b,u) \in I \}$ in which $I = u/\alpha$. Theorem~\ref{thm:grun}(Grun's Lemma) does not require the presence of idempotents but it is mentioned here since it is an important part of the theory of groups with multiple operators which concern the various examples mentioned at the end of Section~\ref{section:3}.

Let us recall the notion of centrality for algebras with finitary operations.

\begin{definition}
For $\alpha,\beta,\gamma \in \Con A$, the relation $C(\alpha,\beta;\gamma)$ holds in $A$ if for all terms $t(\bar{x},\bar{y})$, tuples $\bar{a} \mathrel{\alpha} \bar{b}$, $\bar{c} \mathrel{\beta} \bar{d}$, the implication
\[
t(\bar{a},\bar{c}) \mathrel{\gamma} t(\bar{a},\bar{d}) \longrightarrow t(\bar{b},\bar{c}) \mathrel{\gamma} t(\bar{b},\bar{d})
\]
is true. When $C(\alpha,\beta;\gamma)$ holds we say $\alpha$ \emph{centralizes} $\beta$ modulo $\gamma$ and write $A \vDash C(\alpha,\beta;\delta)$.  
\end{definition}

There are several basic properties of the centralizer relation \cite[Thm 2.19]{shape}; in particular, if $A \vDash C(\alpha,\beta;\gamma_{i})$ for $\{\gamma_{i} : i \in I \}$, then $A \vDash C(\alpha,\beta; \bigwedge_{i \in I} \gamma_{i})$.

\begin{definition}
The TC-commutator of $\alpha,\beta \in \Con A$ is $[\alpha,\beta] := \delta$ if $\delta$ is the smallest congruence such that $A \vDash C(\alpha,\beta;\delta)$.  
\end{definition}

For general algebras, the commutator satisfies only weak properties connected to the ordering of congruences; however, it is closed under arbitrary joins in the first coordinate. This yields the existence of centers for algebras and its characterization in Lemma~\ref{lem:center}.

\begin{definition}
The \emph{centralizer} of $\alpha \in \Con A$ is the largest congruence $\alpha^{\ast} \in \Con A$ such that $[ \alpha^{\ast}, \alpha ] = 0$. The \emph{center} of $A$ is the centralizer $\zeta_{A} := 1^{\ast}_{A}$ of the total congruence. 
\end{definition}

\begin{lemma}\cite[Prop 7.17]{bergman}\label{lem:center}
For any algebra $A$, 
\begin{align}
\zeta_{A} = \{ (a,b) \in A^{2} : \forall t(x,\bar{y}), \forall \bar{c},\bar{d}, A \vDash t(a,\bar{c}) = t(a,\bar{d}) \longrightarrow t(b,\bar{c}) = t(b,\bar{d}) \}.
\end{align}
\end{lemma}

The main role for a Mal'cev operation given centralizing congruences is presented by the following.

\begin{lemma}\label{lem:11}\cite[Thm 3.8]{Kiss}
Let $A \in \mathcal V$ be a Mal'cev variety. Let $p \in \Pol A$ and $\bar{a} \mathrel{\alpha} \bar{b} \mathrel{\beta} \bar{c}$. If $[\alpha , \beta] = 0$, then   
\[
m(p(\bar{a}),p(\bar{b}),p(\bar{c})) = p(m(a_{1},b_{1},c_{1}),m(a_{2},b_{2},c_{2}),\ldots,m(a_{n},b_{n},c_{n})) .
\]
\end{lemma}

The cornerstone of the characterization in Gr\"un's lemma is in certain very useful sets which generate the commutator as a congruence in general Mal'cev algebras. The proof of this for binary and higher commutators constitutes a major part of the development in \cite{supernil}.

\begin{definition}
An operation $f : A^{n} \rightarrow A$ \emph{absorbs} the tuple $(a_{1},\ldots,a_{n})$ to $u \in A$ if $f(x_{1},\ldots,x_{n}) = u$ whenever some $x_{i} = a_{i}$. 
\end{definition}

In the following, we observe that commutators are generated by values of binary polynomials which are absorbing when restricted to coordinate squares.

\begin{lemma}\cite[Lem 6.9]{supernil}\label{lem:comgen1}
Let $A$ be a Mal'cev algebra and $\alpha_{1},\alpha_{2} \in \Con A$. Then $[\alpha_{1},\alpha_{2}]$ is generated as a congruence by the set
\begin{align*}
R(\alpha_{1},\alpha_{2}) &= \{ ( f(b_{1},b_{2}) , f(a_{1},a_{2}) ) : a_{i} \mathrel{\alpha_{i}} b_{i}, f \in \Pol_{2} A \text{ absorbs } (a_{1},a_{2}) \};
\end{align*}
that is, $[\alpha_{1},\alpha_{2}] = \Cg^{A}(R)$. 
\end{lemma}

In the case of principal congruences, the commutator has an especially simple description without the requirement of congruence generation.

\begin{lemma}\cite[Lem 6.13]{supernil}\label{lem:comgen2}
Let $A \in \mathcal V$ be a Mal'cev variety and $a,b,c,d \in A$. Then 
\begin{align*}
[\theta(a,b),\theta(c,d)] &= \{ ( f(b,d) , f(a,c) ) : f \in \Pol_{2} A \text{ absorbs } (a,c) \}.
\end{align*}
\end{lemma}

The following Proposition is a generalization to semiassociative Mal'cev algebras of \cite[Cor 6.12]{supernil} which characterizes the ideal associated to the commutator in expanded groups. The distinction in the generation of the commutator determined by principal congruences is reflected by the absence of the Mal'cev term in generating each commutator block.

\begin{proposition}\label{lem:comgensub}
Let $A$ be a semiassociative Mal'cev algebra. Let $u \in A$.
\begin{enumerate}

	\item If $\alpha_{1},\alpha_{2} \in \Con A$, then $u/[\alpha_{1},\alpha_{2}]$ is generated by the set 
	\begin{align}
	S = \{ c(a_{1},a_{2}) : a_{i} \in u/\alpha_{i}, c \in \Pol_{2} A \text{ absorbs } (u,u) \text{ to } u  \}
	\end{align}
	using the operation $+_{u}$ and $u$.
	
	\item If $a,b,c,d \in A$, then 
	\begin{align}\label{eq:302}
	\begin{split}
	u/[\theta(a,b),\theta(c,d)] &= \{ c(x,y) : x \in u/\theta(a,b),y \in u/\theta(c,d), \\
	&\quad \quad \quad \quad c \in \Pol_{2} A \text{ absorbs } (u,u) \text{ to } u \}
	\end{split}
	\end{align}

\end{enumerate}
\end{proposition}
\begin{proof}

(1) Let us first observe that the set $R(\alpha_{1},\alpha_{2})$ in Lemma~\ref{lem:comgen1} is closed under the unary polynomials and contains the diagonal by taking constants. Take a pair $(f(b_{1},b_{2}) , f(a_{1},a_{2})) \in R(\alpha_{1},\alpha_{2})$ for $f \in \Pol_{2} A$ which absorbs $(a_{1},a_{2})$, then define the binary polynomial $g(x,y) := m(f(a_{1},a_{2}),f(x,y),f(b_{1},b_{2}))$. Then $g \in \Pol_{2} A$ absorbs $(a_{1},a_{2})$ to the value $f(b_{1},b_{2})$ and $g(b_{1},b_{2})=f(a_{1},a_{2})$; thus, $(f(a_{1},a_{2}),f(b_{1},f_{2})) \in R(\alpha_{1},\alpha_{2})$ and so $R(\alpha_{1},\alpha_{2})$ is symmetric. It follows that the congruence generated by $R(\alpha_{1},\alpha_{2})$ is just the transitive closure.

Let $S'$ denoted the subset generated by $S$ with the operation $+_{u}$ and $u$. Take $c(a_{1},a_{2}) \in S$ where $c \in \Pol_{2} A$ absorbs $(u,u)$ to $u$. Then $c(u,a_{2}) = u = c(u,u) \rightarrow c(a_{1},a_{2}) \mathrel{[\alpha_{1},\alpha_{2}]} c(a_{1},u)=u$; thus, $S \subseteq u/[\alpha_{1},\alpha_{2}]$, and because $m$ is idempotent we see that $S' \subseteq u/[\alpha_{1},\alpha_{2}]$.

For the reverse inclusion, first consider $(r_{1},r_{2}) = (f(b_{1},b_{2}),f(a_{1},a_{2})) \in R(\alpha_{1},\alpha_{2})$ where $f$ absorbs $(a_{1},a_{2})$. Define the binary polynomial 
\begin{align*}
h(x,y):= m(f(m(x,u,a_{1}),m(y,u,a_{2})),f(a_{1},a_{2}),u).
\end{align*}
Then $h \in \Pol_{2} A$ absorbs $(u,u)$ to $u$ with $u \mathrel{\alpha_{1}} m(b_{1},a_{1},u)$ and $u \mathrel{\alpha_{1}} m(b_{2},a_{2},u)$. So we have by the identity Eq~\eqref{eqn:1a}
\begin{align*}
m(r_{1},r_{2},u) = m(f(b_{1},b_{1}),f(a_{1},a_{2}),u) = h(m(b_{1},a_{1},u),m(b_{2},a_{2},u)) \in S.
\end{align*}
Now take $(r,u) \in [\alpha_{1},\alpha_{2}]$. Then by the previous comments there is a sequence $r = r_{1},r_{2},\ldots,r_{n-1},r_{n}=u$ such that each $(r_{i},r_{i+1}) \in R(\alpha_{1},\alpha_{2})$. We have seen $m(r_{i},r_{i+1},u) \in S$. Set $s_{i} := m(r_{n-i},r_{n - i + 1},u) \in S$ and consider the element defined by the right-associated composition  
\begin{align*}
v := m \bigg( s_{n-1}, u , \cdots   m \Big( s_{3} , u ,m \big( s_{2} ,u , m ( s_{1},u,r_{n} ) \big) \Big) \cdots \bigg) 
\end{align*}
Repeatedly using the identity Eq~\eqref{eqn:1a}, $s_{i} \in S$ and $r_{n}=u$ we see that $r=r_{1} = v \in S'$; thus, $u/[\alpha_{1},\alpha_{2}] \subseteq S'$.

(2) Follows immediately from (1) where we no longer need to consider congruence generation by the set $R(\theta(a,b),\theta(c,d))$.
\end{proof}

The following corollary is a characterization of the center since in semiassociative Mal'cev algebras congruences are determined by any of the blocks.

\begin{corollary}\label{cor:centchar}
Let $A$ be a semiassociative Mal'cev algebra. Then for all $u \in A$, the congruence block of the center containing $u$ is given by 
\begin{align}\label{eqn:10}
u/\zeta_{A} = \{ a \in A : c \in \Pol_{2} A \text{ absorbs } (u,u) \text{ to } u \rightarrow \forall x \in A, c(a,x) = u \}
\end{align}
\end{corollary}
\begin{proof}
From Lemma~\ref{lem:center}, we can write $\zeta_{A} = \{(a,b) : \forall x, y \in A , [ \theta(a,b), \theta(x,y) ] = 0_{A} \}$. Now follows from Proposition~\ref{lem:comgensub}(2). 
\end{proof}

For groups with multiple operators, the group's neutral element is idempotent and so is the natural choice for $u \in A$ in Proposition~\ref{lem:comgensub} and Corollary~\ref{cor:centchar}; in these cases, we have a useful description of the associated center and commutator ideals.

\begin{definition}
For an algebra $A$ and $u \in A$, let $\mathrm{Ab}^{2}_{u} A = \{ c(x,y) \in \Pol_{2} A : c \text{ absorbs } (u,u) \text{ to } u \}$. 
\end{definition}

Proposition~\ref{lem:comgensub} and Corollary~\ref{cor:centchar} give us the full set of binary absorbing polynomials to generate the commutator or describe the center. For particular algebras of interest, the operations and identities may allows us to make more uniform or simpler choices for $C \subseteq \mathrm{Ab}^{2}_{u} A$ to represent the commutator or center; for example, the singleton $C = \{ [x,y] \}$ in the case of groups and the choice of the neutral element for $u$. By allowing for different possible choices of binary absorbing polynomials in the preamble of Theorem~\ref{thm:grun}, there is an additional utility which allows us to easily recover several previous analogues of Gr\"un's lemma.


\section{Gr\"{u}n's Lemma }\label{section:3}

For any Mal'cev algebra $A$, we define the upper central series in the following manner: set $\zeta_{1} := \zeta_{A}$ and $\zeta_{n+1} \in \Con A$ such that $\zeta_{n+1}/\zeta_{n} := \zeta_{A/\zeta_{n}}$. This produces an ascending series
\begin{align*}
0_{A} \leq \zeta_{A} \leq \zeta_{2} \leq \cdots \leq \zeta_{k} \leq \cdots \leq 1_{A} .
\end{align*}
For the second-center, applying Corollary~\ref{cor:centchar} to $\zeta_{2}(A)/\zeta(A) = \zeta_{A/\zeta(A)}$ yields the following characterization: for all $u \in A$, $a \in u/\zeta_{2}$ if and only if
\begin{align}\label{eqn:11}
\forall c \in \mathrm{Pol}_{2} A , \, c(u/\zeta_{A},A) \subseteq u/\zeta_{A} \wedge c(a,u/\zeta_{A}) \subseteq u/\zeta_{A} \longrightarrow c(a,A) \subseteq u/\zeta_{A} .
\end{align}
We make an additional observation. Fix $u \in A$ and take $c(x,y) \in \mathrm{Ab}^{2}_{u} A$; thus, $c(u,x) = u \wedge c(x,u) = u$. Then for $a \in u/\zeta_{2}$, it follows then that $c$ satisfies $c(u/\zeta_{A},A) \subseteq u/\zeta_{A}$ and $c(A,u/\zeta_{A}) \subseteq u/\zeta_{A}$; thus, by Eq~\eqref{eqn:11} we conclude that 
\begin{align}\label{eqn:111}
c(a,A) \subseteq u/\zeta_{A}    
\end{align}
because $a \in u/\zeta_{2}$.

An algebra is \emph{perfect} if $[1_{A},1_{A}] = 1_{A}$. We are now in a position to prove the general Gr\"un's lemma in semiassociative Mal'cev algebras..

\begin{theorem}\label{thm:grun}
Let $A$ be a perfect semiassociative Mal'cev algebra. Let $u \in A$ and $C, Z \subseteq \mathrm{Ab}^{2}_{u} A$ such that,  
\begin{itemize}

	\item $u/[1_{A},1_{A}]$ is generated by $\{d(a,b) : a,b \in A , d \in C \}$ using the operation $+_{u}$ and $u$, and
	
	\item $u/\zeta_{A} = \{ a : \forall c \in Z, x \in A, c(a,x) = u \}$.

\end{itemize}
The following are equivalent:
\begin{enumerate}

	\item[(G1)] $\zeta_{A} = \zeta_{2}$ ;
	
	\item[(G2)] for all $a \in u/\zeta_{2}$, $c \in Z$, $d \in C$, $c(a,d(x,y)) = u$ and $c(a,x +_{u} y) = c(a,x) +_{u} c(a,y)$.

\end{enumerate}
\end{theorem}
\begin{proof}

(G1) $\Rightarrow$ (G2): Assume $\zeta_{A} = \zeta_{2}$ so that $0_{A/\zeta_{A}} = \zeta_{2}/\zeta_{A} = \zeta_{A/\zeta_{A}}$. Let $a \in u/\zeta_{2} = u/\zeta_{A}$. Take $c(x,y) \in Z$, $d(x,y) \in C$ and note they both absorb $(u,u)$ to $u$. Define polynomial $h(x,y,z):= m( c(x,d(y,z)), d(c(x,y),c(x,z)), u)$; thus, 
\begin{align}\label{eqn:21}
h(a,y,z) + d(c(x,y),c(x,z)) = c(a,d(y,z)) .
\end{align} 
We evaluate
\begin{align*}
h(a,u,y) = m( c(a,d(u,y)), d(c(a,u),c(a,y)), u) &= m(c(a,u),d(u,c(a,y)),u) = u \displaybreak[0]\\
h(u,u,y) = m( c(u,d(u,y)), d(c(u,u),c(u,y)), u) &= m( c(u,u) , d(u,u) ,u) = u \displaybreak[0]\\
h(u,x,y) = m( c(u,d(x,y)), d(c(u,x),c(u,y)), u) &= m( u, d(u,u) , u ) = u . \displaybreak[0]\\
\end{align*}

Since $(a,u,y) \mathrel{\zeta_{A}} (u,u,y) \mathrel{1_{A}} (u,x,y)$ and $a \in u/\zeta_{A}$, we have by Lemma~\ref{lem:11} that 
\begin{align*}
u = m(u,u,u) &= m(h(a,u,y),h(u,u,y),h(u,x,y)) \displaybreak[0]\\
&= h(m(a,u,u),m(u,u,x),m(y,y,y)) = h(a,x,y) . \displaybreak[0]\\
\end{align*}
Then we see from Eq~\eqref{eqn:21} that $c(a,d(y,z)) = d(c(a,y),c(a,z))$. Since $d$ is binary absorbing, for any $x,y \in u/\zeta_{A}$ we have $d(u,u) = u = d(u,y) \rightarrow u= d(x,u) \mathrel{[\zeta_{A},\zeta_{A}]} d(x,y)$; thus, $u = d(x,y)$ since $\zeta_{A}$ is an abelian congruence. Then by Eq~\eqref{eqn:111}, we have $c(a,d(y,z)) = d(c(a,y),c(a,z))=u$ because $a \in u/\zeta_{2}$. A similar calculation for $f(x,y,z) := m(c(x,y +_{u} z),c(x,y) +_{u} c(x,z),u)$ yields $c(a,y) +_{u} c(a,z)$.

(G2) $\Rightarrow$ (G1) To show $\zeta_{2} \leq \zeta_{A}$ it suffices to show $u/\zeta_{2} = u/\zeta_{A}$ since congruences are determined by any block. Take $a \in u/\zeta_{2}$ and let $c \in Z$. We must show that $c(a,x) = u$ for all $x \in A$; in this case, from Eq~\eqref{eqn:10} we conclude that $a \in u/\zeta_{A}$. From the assumptions in (G2), we have $c(a,d(x,y)) =u$ for all $d \in C$ and $c(a,x +_{u} y) = c(a,x) +_{u} c(a,y)$; therefore, $c(a,x) = u$ for all $x \in u/[1_{A},1_{A}] = u/1_{A} = A$ since $u/[1_{A},1_{A}]$ is generated by the set $\{d(x,y) : x,y \in A, d \in C \}$ using the operation $+_{u}$ and constant $u$, and $A$ is perfect.
\end{proof}

In the following corollaries, we recover in a uniform manner some of the unconditional and conditional analogues of Grun's lemma. For each algebra, we use the particular additional identities of the algebra to make simplified choices for the set of binary absorbing polynomials which generate the commutator as in Proposition~\ref{lem:comgensub}(1) or characterize the center as in Corollary~\ref{cor:centchar}.

\begin{corollary}(Gr\"un's Lemma \cite[Satz 4]{Grun})\label{cor:grungrp}
If $G$ is a perfect group, then $Z_{2}(G) = Z(G)$.
\end{corollary}
\begin{proof}
Take $C=Z = \{ [x,y] \}$ and verify condition (G2) in Theorem~\ref{thm:grun} - this is the original calculation. For the choice $u=1$ of the neutral element in the group, $u/\zeta_{G} = Z(G)$ and $a/\zeta_{2} = Z_{2}(G)$ the second-center. For $a \in Z_{2}(G)$, we have $[a,G] \subseteq Z(G)$. Then apply the commutator identity 
\begin{align}\label{eqn:keygrp}
[a,xy] = [a,x][a,y]^{x} = [a,x][a,y] = [a,y][a,x] = [a,yx];
\end{align}
therefore, $[a,[x,y]] = 1$. 
\end{proof}

They key commutator identity for groups used in Eq~\eqref{eqn:keygrp} is one of the defining identities for the bracket operation in multiplicative Lie algebras; therefore, it is reasonable to expect the unrestricted analogue to hold. A \emph{multiplicative Lie algebra} is an algebra $\left\langle A, \cdot, 1, \{~,~\} \right\rangle$ where $\left\langle A, \cdot, 1 \right\rangle$ is a group with neutral 1 and the binary operation $\{~,~\}: A \times A \to A$ satisfies the identities 
	\begin{equation}\label{(2.1)}
		\{x,x\} = 1
	\end{equation}
	\begin{equation}\label{(2.2)}
		\{x , y\cdot z\} = \{x,y\}\cdot {\{x,z\}^y}
	\end{equation}
	\begin{equation}\label{(2.3)} 
		\{ x\cdot y, z\} = {\{y, z\}^x}\cdot \{x, z\}
	\end{equation}
	\begin{equation}\label{(2.4)}
		\{\{x, y\},{z^y}\} \cdot \{ \{ y, z\}, {x^z}\} \cdot \{ \{ z, x\},{y^x}\} = 1
	\end{equation}
	\begin{equation}\label{(2.5)}
		\{x, y\}^z = \{ {x^z},{y^z} \}
	\end{equation}
for all $x,y,z \in A$ where $x^z = z \cdot x \cdot z^{-1}$. These identities imply the additional 
\begin{align}\label{eqn:899}
\{x, [y,z]\} &= [x,\{y,x\}]
\end{align}
\begin{align}\label{eqn:inversebrack}
[x,y]^{-1} &= [y,x]
\end{align}
as pointed out by \cite[pg. 3]{GJ11}.

\begin{corollary}\label{cor:14}
If $A$ is a perfect multiplicative Lie algebra, then $\zeta_{2} = \zeta_{A}$. 
\end{corollary}
\begin{proof}
According to \cite[Prop 2.6]{mLa}, we may take $C = Z = \{ [x,y], \{x,y\} \}$ since the commutator ideal is generated as a subgroup from the group commutators and Lie brackets. For convenience write $\zeta(A) = 1/\zeta$ and $\zeta_{2}(A) = 1/\zeta_{2}$ for the associated ideals. Since both group commutator and Lie bracket absorb the neutral element, we have 
\begin{align}\label{eqn:abs5}
[a,A] \subseteq \zeta(A) \quad \quad \text{ and } \quad \quad \{ a,A \} \subseteq \zeta(A)  
\end{align}
for $a \in \zeta_{2}(A)$. As in the group case, Eq~\eqref{(2.2)} and Eq~\eqref{eqn:abs5} yields $[a,x \cdot y] = [a,x] \cdot [a,y] = [a, y \cdot x]$ and $\{a,x \cdot y\} = \{a,x\} \cdot \{a,y\} = \{a, y \cdot x\}$; thus, $[a,[x,y]]= 1$ and $\{a,[x,y]\} = 1$. Then using Eq~\eqref{eqn:899} we have $[a,\{x,y\}] = \{a,[x,y]\} = 1$. For the last case, we take $z \in \zeta_{2}(A)$ such that $a = z^{y}$. Then using the nonabelian Jacobi identity Eq~\eqref{(2.4)} and Eq~\eqref{eqn:inversebrack} we have
\begin{align*}
\{a,\{ x, y, \} \} = \{z^{y},\{ x, y, \} \} = \{ \{ y , z \} , x^{y} \} \cdot \{ \{ z , x \} , y^{x} \} = 1
\end{align*}
since $\{ z , x \}, \{ y , z \} \in \zeta(A)$. Altogether, condition (2) in Theorem~\ref{thm:grun} is verified. 
\end{proof}




A \emph{skew brace} is an algebra $\left\langle A, \cdot, \circ, \right\rangle$ with group operations $\cdot$ and $\circ$ satisfying 
\begin{align}\label{eqn:897}
x \circ (y \cdot z) = (x \circ y) \cdot x^{-1} \cdot (x \circ z) 
\end{align}
where $x^{-1}$ is the inverse for $\cdot$. It follows that both groups operations share the same unique neutral element $1 \in A$. In the study of skew braces, the term operation 
\begin{align}\label{eqn:766}
x \ast y = x^{-1} \cdot (x \circ y) \cdot y^{-1} 
\end{align}
provides a measure for the ``discrepancy'' between the two groups; for example, $x \ast y=1$ if and only if the group operations are equal and $x \ast y = [x^{-1},y]$ if and only if the group operations are opposites. Directly from Eq~\eqref{eqn:897} and Eq~\eqref{eqn:766} we have the following useful identities 
\begin{align}\label{eqn:200}
a \ast (x \cdot y) &= (a \ast x) \cdot x \cdot (a \ast y) \cdot x^{-1} \\ \label{eqn:201}
(x \circ y) \ast a &= (x \ast (y \ast a)) \cdot (y \ast a) \cdot (x \ast a) .
\end{align}
It is also clear that any two operations from $\{ \cdot, \circ, \ast \}$ define the other operation.

The commutator in skew-braces can be calculated explicitly by the two group commutators $[x,y]$, $[x,y]_{\circ}$ and $x \ast y$. If we write $I = 1/\alpha$ and $J = 1/\beta$ for the ideals associated to $\alpha, \beta \in \Con A$, then the ideal $[I,J]^{A} := 1/[\alpha,\beta]$ associated to the commutator is the ideal generated by the set $\{ [a,b], [a,b]_{\circ} ,a \ast b :  a \in I, b \in J \}$ \cite{SKB}. For $X,Y \subseteq A$, let $[X,Y]$ denote the commutator subgroup and $X \ast Y$ the subgroup generated by $\{ x \ast y :  x \in X, y \in Y\}$ in $\left\langle A, \cdot \right\rangle$, and let $[X,Y]_{\circ}$ denote the commutator subgroup generated in $\left\langle A, \circ \right\rangle$. A rather deep result from \cite[Prop 3.3]{ballister1} is that the commutator can also be computed by
\begin{align}\label{eqn:commeq}
[A,A]^{A} =  (A \ast A)  \cdot [A,A] =  (A \ast A) \circ [A,A]_{\circ} .
\end{align}
The next result essentially follows \cite[Thm A]{many1} which itself relies on \cite[Cor 2.6]{Tsang} but introduces calculations with trifactorized groups; however, the present argument  highlights the importance of the derivations in \cite{Tsang} and the specialized representation of the commutator in \cite[Prop 3.3]{ballister1}. Set $Z(A) = 1/\zeta_{A}$ and $Z_{2}(A) = 1/\zeta_{2}$ for the ideals associated to the center and second-center, respectively.

\begin{corollary}\label{cor:perfgrun}(see \cite[Thm A]{many1}, \cite[Cor 1.8]{Tsang})
Let $\left\langle A \cdot , \circ \right\rangle$ be a perfect skew brace. The following are equivalent:
\begin{enumerate}

	\item $Z_{2}(A) = Z(A)$;
	
	\item for all $a \in Z_{2}(A)$, the map $i_{a}(x) = a \circ x \circ \bar{a}$ is an automorphism of $\left\langle A, \cdot \right\rangle$ where $\bar{a}$ is the inverse in $\circ$; 
	
	\item $Z_{2}(A)$ is contained in the center of the group $\left\langle A, \circ \right\rangle$;
	
	\item for all $a \in Z_{2}(A)$, $[a,x \cdot y]_{\circ} = [a,x]_{\circ} \cdot [a,y]_{\circ}$ and 
\begin{align*} 
[Z_{2}(A),[A,A]]_{\circ} = [Z_{2}(A),A \ast A]_{\circ} = 1.
\end{align*}
\end{enumerate}
\end{corollary}
\begin{proof}
$(1) \Rightarrow (4)$: Assume $Z_{2}(A) = Z(A)$. Take $C=Z = \{[x,y],[x,y]_{\circ}, x \ast y\}$ in Theorem~\ref{thm:grun}. Then condition (G2) applied to the commutator $[x,y]_{\circ}$ yields the left-distributivity of $[a,x]_{\circ}$ over $x \cdot y, [x,y]$ and $x \ast y$ for $a \in Z_{2}(A)$. The statements in (4) follow.

$(4) \Rightarrow (3)$: Using the conditions in (4) and the representation of the commutator in Eq~\eqref{eqn:commeq}, we can write
\begin{align*}
[Z_{2}(A), A]_{\circ} = [Z_{2}(A), [A,A]^{A}]_{\circ} = [Z_{2}(A), (A \ast A) \cdot [A,A]]_{\circ} = 1
\end{align*}
because $A$ is perfect; thus, elements in $Z_{2}(A)$ commute with the $\circ$ operation.

$(3) \Rightarrow (2)$: In this case, we have $i_{a}(x) = \id$ for all $a \in Z_{2}(A)$.

$(2) \Rightarrow (1)$: We argue the same as in \cite[Prop 5.3]{many1}. This implication relies principally on \cite[Cor 2.6]{Tsang} where it shown that 
\begin{align}\label{eqn:Tsang1}
\forall a \in Z_{2}(A) \, , \, i_{a} \in \Aut \left\langle A, \cdot \right\rangle \quad \Leftrightarrow  \quad (A \ast A) \ast Z_{2}(A) = 1 
\end{align}
for any skew braces. It is also shown that for any skew braces we always have 
\begin{align}\label{eqn:a1}
Z_{2}(A) \ast (A \ast A) &= 1 &(\text{ \cite[Prop 2.2]{Tsang} }) \\ \label{eqn:a2}
[ Z_{2}(A), A \ast A ] &= 1 &(\text{ \cite[Prop 2.3]{Tsang} }) 
\end{align}
For $a \in Z_{2}(A)$, we have $a \ast A, A \ast a \subseteq Z_{A}$; thus, Eq~\eqref{eqn:200} and Eq~\eqref{eqn:201} take the form $a \ast (x \cdot y) = (a \ast x) \cdot (a \ast y)$ and $(x \circ y) \ast a = (x \ast a) \cdot (y \ast a)$ which imply
\begin{align}\label{eqn:3456}
Z_{2}(A) \ast [A,A] &= 1 \\ \label{eqn:3457}
[A, A]_{\circ} \ast Z_{2}(A) &= 1 .
\end{align}
By the representations in Eq~\eqref{eqn:commeq}, using Eq~\eqref{eqn:a1} and Eq~\eqref{eqn:3456} in the first line, and Eq~\eqref{eqn:Tsang1} and Eq~\eqref{eqn:3457}  in the second line we conclude that 
\begin{align*}
Z_{2}(A) \ast A &= Z_{2}(A) \ast [A,A]^{A} = Z_{2}(A) \ast ( (A \ast A) \cdot [A,A] ) = 1 \\
A \ast Z_{2}(A) &= [A,A]^{A} \ast Z_{2}(A) = ( (A \ast A ) \circ [A,A]_{\circ} ) \ast Z_{2}(A) =1 . 
\end{align*}
Using the group calculation $[Z_{2}(A), [A, A]] = 1$ and Eq~\eqref{eqn:a2} we see that elements in $Z_{2}(A)$ commute with the $\cdot$ operation. Since it follows from Eq~\eqref{eqn:766} that $\{a \in A : a \ast x = 1 = x \ast a \} \cap Z(A, \cdot) \subseteq Z(A,\circ)$, we have shown $Z_{2}(A) \subseteq Z(A)$. 
\end{proof}


\section{Discussion}\label{sec:discussion}

We know that groups and multiplicative Lie algebras satisfy an unconditional Gr\"un's lemma, but skew braces do not. Is it possible to arrive at a more general understanding or characterization of algebras which do ?

\begin{problem}
Which perfect semiassociative Mal'cev algebras satisfy an unconditional Gr\"un's lemma ?  
\end{problem}

Corollary~\ref{cor:perfgrun}(2) illustrates that further refinements of the characterization in Theorem~\ref{thm:grun}(G2) for particular algebras are possible and most likely require a deep understanding of the available identities and representations of the commutator.

\begin{problem}\label{prob:2}
For those perfect semiassociative Mal'cev algebras which do not satisfy an unconditional Gr\"un's lemma, develop a refined characterization of when the upper-central series stabilizes; in particular, when the second-center and center coincide ?  
\end{problem}

According to J.D.H. Smith \cite[Thm 4.3]{smith}, conformal algebras $C$ correspond to two-sorted conformal algebras of the form $(C,\mathds{Z})$, and the class of two-sorted conformal algebras are equivalent to a Mal'cev variety of equational conformal algebras \cite[Thm 5.3, Prop 5.4]{smith} ; in particular, they are expansions of an abelian group operation and so are examples of semiassociative Mal'cev algebras for which Theorem~\ref{thm:grun} applies. We may consider the above problems for conformal algebras in terms of their equivalent formulation as equational conformal algebras.

\begin{problem}
Do conformal algebras satisfy an unconditional Gr\"un's lemma ? If not, give a refined characterization of when the second-center and center coincide for perfect conformal algebras. 
\end{problem}

Two-sorted vertex algebras do not form a Mal'cev variety of two-sorted algebras \cite[Thm 6.8]{smith}; nevertheless, they are constructed from sorts of abelian groups and certain lattice-ordered rings. Is it possible to adapt to vertex algebras the results of this manuscript ?

\begin{acknowledgments}
The research in this manuscript was supported by NSF China Grant \#12071374.
\end{acknowledgments}


\end{document}